\documentclass[]{imsart}

\RequirePackage{amsthm,amsmath,amsfonts,amssymb}
\RequirePackage[numbers]{natbib}
\RequirePackage[colorlinks,citecolor=blue,urlcolor=blue]{hyperref}
\RequirePackage{graphicx}% uncomment this for including figures
\usepackage{verbatim}

\startlocaldefs
\theoremstyle{plain}
\newtheorem{thm}{Theorem}
\newtheorem{cor}{Corollary}
\newtheorem{lem}{Lemma}
\newtheorem{prop}{Proposition}
\theoremstyle{remark}
\newtheorem{remark}{Remark}
\newtheorem{definition}{Definition}
\newcommand{\bmu}{\boldsymbol{\mu}}
\newcommand{\bnu}{\boldsymbol{\nu}}
\newcommand{\C}{\mathbf{C}}

\newcommand{\CBI}{\operatorname{CBI}}
\newcommand\cev[1]{\overleftarrow{#1}}
\newcommand{\cl}{\operatorname{\mathbf{cl}}}
\newcommand{\D}{\mathbb{D}}
\newcommand{\downto}{\downarrow}
\newcommand{\E}{\mathbb{E}}
\newcommand{\eps}{\varepsilon}
\newcommand{\f}{\mathfrak{f}}

\newcommand{\GWI}{\operatorname{GWI}}
\newcommand{\h}{\operatorname{\mathbf{ht}}}
\newcommand{\HH}{\mathbf{H}}
\newcommand{\J}{\mathbf{J}}
\newcommand{\LL}{\mathbf{L}}
\newcommand{\MCB}{\operatorname{MCB}}
\newcommand{\MCBI}{\operatorname{MCBI}}
\newcommand{\N}{\mathbb{N}}
\newcommand{\prej}{\overset{(j)}{\prec}}
\newcommand{\R}{\mathbb{R}}
\newcommand{\tFrak}{\mathfrak{t}}
\newcommand{\T}{\mathcal{T}}
\newcommand{\U}{\mathbf{U}}
\newcommand{\upto}{\uparrow}
\newcommand{\X}{\mathbf{X}}
\newcommand{\Y}{\mathbf{Y}}
\newcommand{\Z}{\mathbf{Z}}
\endlocaldefs

\begin{document}

\begin{frontmatter}
%%%%%%%%%%%%%%%%%%%%%%%%%%%%%%%%%%%%%%%%%%%%%%
%%                                          %%
%% Enter the title of your article here     %%
%%                                          %%
%%%%%%%%%%%%%%%%%%%%%%%%%%%%%%%%%%%%%%%%%%%%%%
\title{Encoding multitype Galton-Watson forests and a multitype Ray-Knight theorem}
%\title{A sample article title with some additional note\thanksref{T1}}
\runtitle{Multitype Ray-Knight theorem}
%\thankstext{T1}{A sample of additional note to the title.}

\begin{aug}
\author[A]{\fnms{David} \snm{Clancy, Jr.}\ead[label=e1]{djclancy@uw.edu}},

%%%%%%%%%%%%%%%%%%%%%%%%%%%%%%%%%%%%%%%%%%%%%%
%% Addresses                                %%
%%%%%%%%%%%%%%%%%%%%%%%%%%%%%%%%%%%%%%%%%%%%%%
\address[A]{University of Washington, \printead{e1}}

\end{aug}

\begin{abstract}
		We provide a simple forest model to encode the genealogical structure of a multitype Galton-Watson process with immigration. We provide two encodings of these forests by stochastic processes. We show, under appropriate conditions, the depth-first encodings of each particular type converge to a solution to a system of stochastic integral equations involving height processes perturbed by functionals of their local times. The forest picture allows us to extend the Ray-Knight theorem and show that local time of the solution to the system of equations form a multitype continuous state branching process with immigration. These assumptions underlying our weak convergence arguments are easily seen to be met in the Brownian setting, and more generally an $\alpha$-stable setting for any $\alpha\in(1,2]$.
\end{abstract}



\end{frontmatter}

%%%%%%%%%%%%%%%%%%%%%%%%%%%%%%%%%%%%%%%%%%%%%%
%%%% Main text entry area:

	\section{Introduction}
	
	Galton-Watson processes are simple models for population evolution: each individual in the population gives birth to a random number of children each with the same law. By considering the entire history of the population, we often view these processes as encoded by an underlying random forest, i.e. the collection of family trees. In his continuum random tree (CRT) trilogy Aldous \cite{Aldous_CRT1, Aldous_CRTII,Aldous.93} shows that certain Galton-Watson trees conditioned on having total progeny $n$ possesses a metric space scaling limit as $n\to\infty.$ Since then, much work has been done on the scaling limits of unconditioned versions of Galton-Watson forests, with the possible inclusion of some immigration term. See \cite{LL_BPLP, LL_BPLP2, AP_BBAss, Pitman_SDE, DL_Levy, Duquesne_CRTI} among others.
	
	Multitype Galton-Watson processes and forests are a simple extension of single type Galton-Watson processes. They describe the evolution of a population of individuals made up of different types and have been suggested as model of various biological phenomena including cell mutation, see for example \cite{KA.15}. Limiting genealogical structure for these models have not yet been developed in general, although some progress has been made. In \cite{Miermont_multiType}, Miermont shows under certain assumptions including criticality and finite variance a multitype Galton-Watson forest behaves asymptotically much like a Galton-Watson forest after ignoring types. That is, he shows that the Harris path on the multitype Galton-Watson forest has a scaling limit which is a reflected Brownian motion, which is the scaling limit of the Harris path for certain Galton-Watson forests of a single type \cite{LG_RandomTreeApps}. The stable analog of this result was obtained in \cite{BO_AlphaStablemultitype}. Both of these works ignore types. See also \cite{dR_infiniteTypes} for a model with infinitely many types, and \cite{HS.19} for a model where certain random multitype trees have metric space scaling limits. The purpose of this article is to describe one way in which we can obtain scaling limits for a collection of Harris paths where we distinguish individuals by type. 
	
	A multitype branching process with immigration and with $N$ types, say $Z = (Z^1,\dotsm, Z^N)$, is a Markov chain on $\{0,1,\dotsm\}^N$ such that 
	\begin{equation*}
	Z^j(h+1) = \left(\sum_{i=1}^j \sum_{\ell = 1}^{Z^i(h)} \xi_{h,\ell}^{i, j} \right)  + \eta_{h}^{j},
	\end{equation*} where for each $i$ the random vectors $(\xi^{i}_{h,\ell}:= (\xi^{i,1}_{h,\ell}, \dotsm \xi^{i,N}_{h,\ell}); h\ge 0, \ell\ge 1)$ are i.i.d. and $(\eta_h  = (\eta_h^1,\dotsm, \eta_h^N); h\ge 0)$ are i.i.d. and these vectors are independent of each other. There can be some correlation between the coordinates of the vector; however, our results will assume that the coordinates are independent. The vectors $\xi^{i}_{h,\ell}$ describe the types of number and types of children that individuals of type $i$ give birth to in the population and the vectors $\eta_h$ describe immigration from an outside source. We will encode these multitype forests in two ways. One is in a breadth-first manner which captures information on the branching structure \cite{CPU_affine,CGU_CSBPI} and the other is depth-first manner which encodes certain genealogical information \cite{LG_RandomTreeApps, Duquesne_CRTI}. Both of these encodings can be viewed as an extension of the {\L}ukasiewicz path and is based on a similar encoding in the work of Chaumont and Liu in \cite{CL_MultiType} where the authors encode multitype forests without immigration in a terms of a family of $N\times N$ random walks with increments related to the random variables $\xi^{i,j}_{h,\ell} - 1_{[i=j]}$ and prove an extension of the Otter-Dwass formula on the total population size of each type. See also \cite{AH.20}.

	Continuous time analogs of the random walk encodings were obtained in \cite{CPU_affine}, see equations \eqref{eqn:dtc} and \eqref{eqn:lampertiGen} below, where random walks are replaced with spectrally positive L\'{e}vy processes or subordinators. In particular they show that any continuous state multitype branching processes with immigration $\Z = (\Z_t = (\Z_t^1,\dotsm, \Z^N_t);t\ge 0)$ is a multiparameter time-change of a collection of $N+1$ many $\R^N$-valued L\'{e}vy processes. Based on this continuous time extension, Chaumont and Marolleau \cite{CM.20} develop some fluctuation theory for certain types of random fields obtained from spectrally positive L\'{e}vy processes. These random fields were called spectrally positive, additive L\'{e}vy fields and appear quite naturally in the study of extinction times for multitype continuous state branching processes without immigration.

	Before turning to some of our results, let us briefly look at the work of Abraham and Mazliak \cite{AM_BPofBM} to formulate the particular form that our stochastic integral equations will satisfy. Expanding on previous works on the branching properties of Brownian paths, they show that a weak solution to the stochastic differential equation
\begin{equation*}
	dZ_v = 2\sqrt{Z_v}dW_v + \Delta'(v)\,dv, \qquad Z_0 = 0
\end{equation*} for sufficiently nice functions $\Delta:\R_+\to \R_+$ can be obtained by looking at the local time of the process
\begin{equation*}
	H_t = \beta_t + \Delta^{-1}(\ell_t)
\end{equation*} where $\beta  = (\beta_t;t\ge 0)$ is a reflected Brownian motion and $\ell = (\ell_t;t\ge 0)$ is its local time at time $t$ and level $0$. The authors give a nice tree picture capturing the underlying branching structure, which is also described in \cite[Chatper 6]{Pitman_CombStoch}. For another description and proof of this result, see \cite{MY_BM}. The calculus of continuous functions of finite variations allows us rewrite the process $H$ as the solutions to the following integral equation:
\begin{equation}\label{eqn:AMH}\left\{
	\begin{array}{l}
		H_t = \beta_t + J_t\\
		\displaystyle J_t = \int_0^t \frac{1}{\Delta'(J_s)}\,d\ell_s
	\end{array}\right..
\end{equation} See \cite[Proposition 0.4.6]{RY}. The characterization above is what we wish to generalize to a system of $N$ stochastic processes representing different types of the multitype branching processes. For more results on the local time of a Brownian motion with drift depending on its local time see, for example, \cite{LPW_RKthm, Le_Note, BP_RKThm} as well as Section 5.2 in \cite{PW_LogisticGrowth}.

\subsection{Assumptions and Statement of Results} \label{sec:statementOfResults}

We will consider systems based on $N$ independent L\'evy processes $\X^1,\dotsm,\X^N$. In order for our proof techniques to follow through, we must make a few technical assumptions. We assume that for each $j\in [N] :=\{1,\dotsm,N\}$ 
\begin{equation*}
	\E\left[\exp\left\{ -\lambda \X_1^j\right\} \right] = \exp\{\psi_j(\lambda)\}
\end{equation*} where 
\begin{equation}\label{eqn:psijs} \begin{array}{l}
		\displaystyle \psi_j(\lambda) = \alpha_j\lambda+ \beta_j \lambda^2  + \int_{(0,\infty)}(e^{-\lambda r}-1-\lambda r1_{[r<1]})\,\pi_j(dr)\\
		\displaystyle \int_1^\infty \frac{1}{\psi_j(u)}\,du <\infty
	\end{array}
\end{equation} with $\alpha_j,\beta_j \ge 0$ and $\pi_j$ is a Radon measure on $(0,\infty)$ with $\int_{(0,\infty)} (r\wedge r^2)\,\pi_j(dr) <\infty$. We say that $(-\psi_j$) is the Laplace exponent of $\X^j$. The conditions above are the technical conditions of \cite{DL_Levy} which guarantee certain nice properties of height processes, which will be discussed later in Section \ref{sec:hp}.

We also must make an assumption that these L\'evy processes can arise under the \textit{same scaling regime}, which we now describe. We say that a probability measure on $\N$ is (sub)critical if $\sum_{k\ge 0} k\mu(k)\le 1$. We assume that there exists sequences of (sub)critical probability measures $(\mu^{j,p};j\in [N],p\ge 1)$ on $\N_0 = \{0,1,\dotsm\}$ with generating functions $g^{j,p}$. Recursively define $g^{j,p}_n$ by $g^{j,p}_n = g^{j,p}_{n-1}\circ g^{j,p}$ with $g^{j,p}_0 = id$. For each $j$ and $p$, let $(\xi^{j,p}_\ell;\ell \ge 0)$ be independent random variables with common law $\mu^{j,p}$. We make the assumption that there exists an increasing sequence $(\gamma_p; p \ge 1)$ of non-negative integers such that for each $j\in [N]$
\begin{equation}\label{eqn:assMu}\begin{array}{l}
		\displaystyle \frac{1}{p} \sum_{\ell = 1}^{p\gamma_p} (\xi^{j,p}_\ell -1) \Longrightarrow \X^j_1\\
		\displaystyle \liminf_{p\to \infty }g^{j,p}_{[\delta \gamma_p]} (0) >0,\qquad \forall \delta>0
	\end{array}.
\end{equation}

\begin{definition}
	We call a family of Laplace exponents $(\psi_j)$ for $j\in [N]$ \textbf{admissible} if they satisfy \eqref{eqn:psijs} and can their corresponding L\'evy processes can be obtained by \eqref{eqn:assMu}.
\end{definition}

\begin{remark}
	The definition of admissible may seem to be quite restrictive; however, Theorem 2.3.2 in \cite{DL_Levy} implies that $\psi_j(\lambda) = c_j\lambda^\alpha$ for $\alpha\in(1,2)$ and $c_j>0$ are admissible. Also, the Lindeberg-Feller CLT implies that $\psi_j(\lambda) = \alpha_j\lambda + \beta_j\lambda^2$ are admissible for $\alpha_j\ge 0,$ $\beta_j>0$. Moreover, anytime \cite[Corollary 2.5.1]{DL_Levy} applies, we can find a family of admissible $\psi_j$'s. 
\end{remark}
As mentioned above, if $(\psi_j;j\in[N])$ are admissible, then there exists a $\psi_j$-height processes $\HH^j = (\HH^j_t;t\ge 0)$ for each $j\in[N]$. The value $\HH^j_t$ measures in a local time sense the size of the set 
\begin{equation} \label{eqn:ltForH}\{s\le t: \inf_{r\in[s,t]} \X^j_r = \X^j_{s-} \}
\end{equation} More details of this process will be discussed in Section \ref{sec:hp}; however, we leave a full discussion of these processes to \cite{DL_Levy}. Due to integral assumption in \eqref{eqn:psijs}, there exists a continuous modification of the process $\HH^j$ \cite[Theorem 1.4.3]{DL_Levy}, and from now on we will only consider this modification.

We let $\ell^j = (\ell^j_t;t\ge0)$ denote the local time at time $t$ and level $0$ of the process $\HH^j$. It is defined as the $L^1$-limit (see \cite[Lemma 1.3.2]{DL_Levy}):
\begin{equation}\label{eqn:ltLim}
	\lim_{\eps\downto 0} \frac{1}{\eps} \int_0^t 1_{[\HH^j_s\le \eps]}\,ds = \ell_t^j = -\inf_{s\le t} \X^j_s.
\end{equation} We also let $\tau^j_x = \inf\{t: \ell^j_t>x\}$ be the right-continuous inverse of $\ell^j$.

We wish to generalize equation \eqref{eqn:AMH} to systems of equations. To do this, we fix $\delta_j> 0$, $x_j\ge 0$ for $j\in [N]$ and $\alpha_{i,j}\ge 0$ for $i\neq j$. We study the following system of stochastic equations
\begin{equation}\label{eqn:HSDE}\left\{
	\begin{array}{l}
		\cev{\HH}^j_t  = \HH^j_t + \J^j_t\\
		\displaystyle \J^j_t = {\int}_{\tau_{x_j}^j}^{t\vee \tau_{x_j}^j} \left\{\delta_j + \sum_{i\neq j} \alpha_{i,j} \LL_\infty^{\J^j_s}(\cev{\HH}^i) \right\}^{-1}\,d\ell_s^j
	\end{array} \right.,
\end{equation} where $\LL_\infty^v(\cev{\HH}^j)$ is the terminal local time of the process $\cev{\HH}^j$. More precisely we prove the following theorem:
\begin{thm}\label{thm:existH}
	Fix any $\delta_j>0$, $x_j\ge 0$ for $j\in[N]$ and any $\alpha_{i,j}\ge 0$ for $i\neq j\in[N]$. Then for any admissible family of Laplace exponents $(\psi_j)_{j\in[N]}$, and independent $\psi_j$-height processes $\HH^j$ there exists a weak solution to the stochastic equation \eqref{eqn:HSDE}. Moreover, in the solution the processes $(\LL_\infty^v(\cev{\HH}^j);v\ge 0)$ are c\`adl\`ag and, almost surely, for any continuous function $g$ with compact support
	\begin{equation*}
		\int_0^\infty g\left(\cev{\HH}^j_t \right)\,dt = \int_0^\infty g(v) \LL_\infty^v(\cev{\HH}^j)\,dv
	\end{equation*}
\end{thm} 

From now on, when we discuss solutions to the equation \eqref{eqn:HSDE} we will implicitly assume that are speaking about the solution given by Theorem \ref{thm:existH}. We can state the following corollary to Theorem \ref{thm:existH}, which is just a simple case where we know a family of $\psi_j$ are admissible. Namely, when the height processes are reflected Brownian motions with negative drift.

\begin{cor}\label{cor:bm}
	Suppose that $\delta_j>0$, $x_j\ge 0$ for $j\in[N]$ and $\alpha_{i,j}\ge 0$ for all $i\neq j\in[N]$. We suppose $\alpha_{j,j}\le 0$. Let $B^j_t$ denote independent standard Brownian motions. Then, for any $\beta_j>0$ there exists a solution to \eqref{eqn:HSDE} for the height processes defined by
	\begin{equation*}
		\HH^j_t = \frac{1}{\beta_j} \left(\sqrt{2\beta_j}B^j_t + \alpha_{j,j} t - \inf_{s\le t} \{ \sqrt{2\beta_j}B^j_s +\alpha_{j,j}s\} \right).
	\end{equation*}
\end{cor}

In order to prove Theorem \ref{thm:existH}, we construct a random multitype forest where the solution arises fairly naturally. The model is discussed in Section \ref{sec:forest}. The forest model also gives rise to a Ray-Knight type theorem, which requires more information on the classification of multitype continuous state branching processes. In lieu of discussing that here, we discuss the Ray-Knight theorem in the context of Corollary \ref{cor:bm} where stochastic calculus can be used and reserve the more general result for Proposition \ref{prop:mcbi}.
\begin{cor}\label{cor:zsde}
	Suppose the hypothesis of Corollary \ref{cor:bm} and for each $j\in[N]$ let $\Z_v^j = \LL_\infty^v(\cev{\HH}^j)$. Then $\Z = (\Z^1,\dotsm, \Z^N)$ is a weak solution to the system stochastic differential equations
	\begin{equation*}
		%\label{eqn:Zsde}
		d\Z_v^j = \sqrt{2\beta_j\Z_v^j}\,dW^j_v + \left\{\delta_j + \sum_{i=1}^N \alpha_{i,j} \Z^i_v \right\}\,dv,\qquad \Z^j_0 = x_j,
	\end{equation*} where $(W^1,\dots, W^N)$ is an $\R^N$-valued Brownian motion. 
\end{cor}
\begin{remark}
	We remark that in the Brownian case of Corollary \ref{cor:bm}, the solution $\cev{\HH}^j$ is a semi-martingale with quadratic variation $\langle\cev{\HH}^j\rangle_t = \frac{2t}{\beta_j}$. Hence, if $L_\infty^v(\HH^j)$ is the semi-martingale local time then $L_\infty^v(\cev{\HH}^j) = \frac{\beta_j}{2} \LL^v_\infty(\cev{\HH}^j)$. Thus there is a similar statement to Corollary \ref{cor:zsde} where we replace $\LL$ with $L$. In this situation, we can also replace the constants in equation \eqref{eqn:HSDE} in order to have $\J^j_t$ depend on $L_\infty^{\J^j_s}(\cev{\HH}^i)$.
	
\end{remark}

\subsection{Organization of the Paper}

In Section \ref{sec:forest} we discuss the forest model and establish the notation that will be used throughout the sequel. This section contains within it many identities that happen in the discrete which will be crucial in our later weak convergence arguments. Section \ref{sec:forestRandom} describes the randomization of the model, where the discrete processes defined on forests become well-known discrete stochastic processes.

In Section \ref{sec:bp} we give a brief overview of (multitype) continuous state branching processes along with their connections to L\'evy processes. We do discuss their characterization as affine processes; however, all the results refer to them through the useful time-change found in \cite{CPU_affine}. Section \ref{sec:hp} also gives a brief overview of the $\psi$-height processes, under the assumptions discussed in Section \ref{sec:statementOfResults}. We do point to references where weaker assumptions are made on the Laplace exponents $\psi$.

Section \ref{sec:ws} contains the main weak convergence arguments. We recall several results that follow from the existing literature that will be useful to state explicitly in this paper. Section \ref{subsec:cevH} is devoted to proving Theorem \ref{thm:cevHconv} from which Theorem \ref{thm:existH} follows from simple observations involving calculus of finite variation functions. We then state the general Ray-Knight theorem with Proposition \ref{prop:mcbi}. In Section \ref{sec:con}, we show how Corollary \ref{cor:bm} follows from Theorem \ref{thm:existH} and how Corollary \ref{cor:zsde} follows from Proposition \ref{prop:mcbi}.

The weak convergence arguments presented in Section \ref{sec:ws} are relatively standard in the random tree literature. Apart from some arguments dealing with time-changes and first passage times, the restriction to admissible $(\psi_j)$ simplifies much of the more delicate arguments. In particular this restriction allows us to use the weak convergence results presented in \cite[Chapter 2]{DL_Levy}. Therefore, the forests description presented in Section \ref{sec:forest} along with the many discrete identities are a large part of the sequel.

%In Section \ref{sec:F}, we prove Theorem \ref{thm:F} and prove various results about solutions to the equation in \eqref{eqn:Fsde}.

%In Section \ref{sec:gs} we discuss a nice consequence of these theorems. We briefly discuss it here as well. Using random matrix theory, Gorin and Sholnikov \cite{GS_paper} have shown a distributional identity of the difference between the area under a Brownian excursion and the integral of the square of its total local time. This was then proved using different techniques by Hariya \cite{Hariya}, and extended to reflected Brownian bridges by Gaudreau Lamarre and Shkolnikov in \cite{GLS_RBB}. This was later extended by the author in \cite{C_GS} to include $\psi$-height processes where $\psi$ satisfies \eqref{eqn:psijs}. We extend this identity to the solutions $\cev{\HH^j}$ to \eqref{eqn:HSDE}. This is contained in Theorem \ref{thm:gs}.

\section{Descriptions of Forests}\label{sec:forest}

In this section we describe how our discrete forests are constructed, and what processes on the forest we define. A forest, say $\f$, will be both rooted and colored by the \textit{colors} or \textit{types} $0,1,\dotsm,N$. The type zero vertices will be of a special kind when we randomize the model. They keep track of the immigration terms. For the colors $1,\dotsm, N$ we will define two separate labelings on the vertices corresponding to a breadth-first ordering and a depth-first ordering. Many of the processes we define are described for 1 type by Duquesne in \cite{Duquesne_CRTI} using a similar forest model.

\subsection{Basic Definitions}

We define a forest $\f$ as a locally finite graph on some (possibly infinite) vertex set $V\subset \N$ which has a finite number of connected components which are themselves rooted planar trees. Recall locally finite means that each vertex has finitely many adjacent vertices. We note that the vertex set $V$ is a subset of the natural number. Therefore, the vertices have some ordering inherited from the natural numbers. To distinguish this ordering from the ordering on the natural numbers we write $\prec$ instead of $<$, i.e for $v,w\in \f$, we say $v\prec w$ if, as natural numbers, $v<w$. A priori, this has no special significance for the forests.

Each connected component is equipped with the graph distance. In each of the connected components, say $\T_1,\dots, \T_M$, there will be a distinguished vertex $\rho_j\in\T_j$ called the root of the tree. 

The height of a vertex $v\in\T_j$ is defined as the distance to the root $\rho_j$, and the height of $v$ is denoted by either $\h(v;\f)$ or $\h(v;\T_j)$ which, depending on the context, should be clear. We also assume that the labeling of the vertices by elements of $\N$ obeys the following ordering property, which is possible by the local finiteness condition:
\begin{enumerate}
	\item[(O1)] If $\h(v;\f)<\h(w;\f)$ then $v\prec w$. 
\end{enumerate} The condition (O1) roughly tells that the ordering $\prec$ describes a breadth-first ordering of the vertices of the trees. In order to have the breadth-first ordering exactly correspond to the ordering $\prec$, we would need a way to order vertices at the same height. As it will turn out, when we randomize our model the exchangeability of the random variables makes specifying the ordering with that specificity irrelevant. 

 For each vertex $v\in \T_j\setminus\{\rho_j\}$, there is a unique adjacent vertex $w$ which lies on the unique geodesic connecting $v$ to $\rho_j$. We call this vertex $w$ the parent of $v$, and is denoted by $\pi(v) = w$. The vertex $v$ is called a child of $w$. 

Next there will be some coloring of the vertices of $\f$, which is just some map $\cl:\f\to \{0,1,\dots,N\}$ subject to the three conditions:
\begin{enumerate}
	\item[(C1)] If $\cl(v) = 0$ then $\cl(\pi(v)) = 0$ or $v$ is a root. 
	\item[(C2)] For each $j\in \{1,\dotsm,N\}$ the connected components of the subgraph induced by the vertices $\cl^{-1}(j) = \{v\in \f: \cl(v) = j\}$ are finite. 
	\item[(C3)] If $v$ and $w$ are such that $0<\cl(v)<\cl(w)$ and $\pi(v) = \pi(w)$ then $v\prec w$. If $\cl(w) = 0$ and $\cl(v)\neq 0$ with $\pi(v) = \pi(w)$ then $v\prec w$ as well. 
\end{enumerate} 

Forests described above equipped with a coloring function satisfying (C1-C3) will be called \textit{colored forests}. We briefly describe what these conditions mean. Condition (C1) tells us that the vertices of color 0, have parents of color 0. In Figure \ref{fig:bfalllabeling}, this means that the non-root red vertices have parents which are also red. Condition (C2) states that the subtrees of $\f$ where all vertices are of a type $j\neq 0$ must be finite. This will be vital when describing the height processes and corresponds to a (sub-)criticality assumption on a branching mechanism. Condition (C3) just guarantees the children of a vertex have some relationship to the ordering of all vertices on the tree. Moreover, conditions (C3) implies that there is some breadth-first manner to which the vertices are labeled. See Figure \ref{fig:bfalllabeling} for an example of a colored forest with 3 colors. 

\begin{figure}
	\centering
	\includegraphics[width=\linewidth]{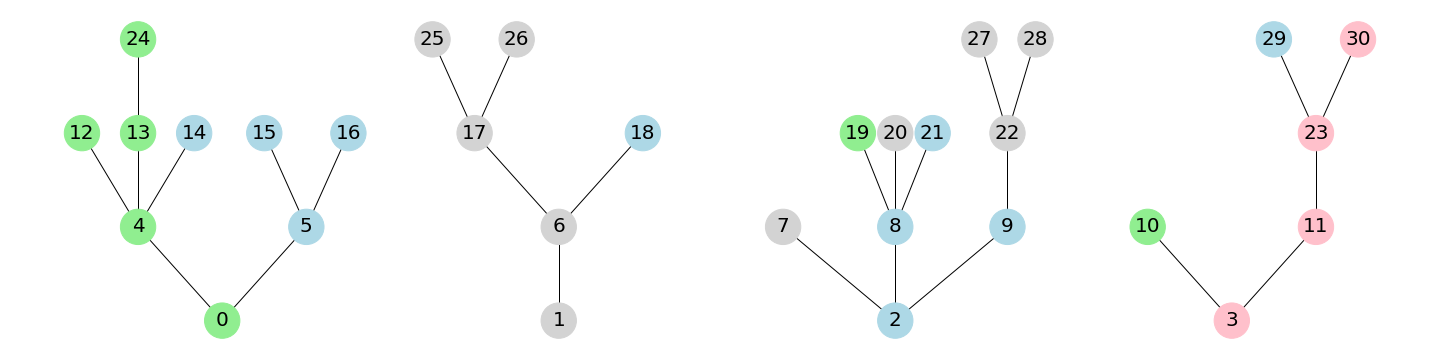}
	\caption{A colored forest with 31 vertices and 3 non-zero colors. Color 1 is depicted in green, color 2 is depicted in grey, color 3 is depicted in blue and color 4 is depicted in red. The vertices are labeled in a with their corresponding integers. The $j$-roots $v^1_0,\dotsm$ depicted the roots of trees of type 1, are labeled $v^1_0 = 0$, $v^1_1 = 10$ and $v_2^1 = 19$.}
	\label{fig:bfalllabeling}
\end{figure}

For each vertex $v$ we define $\chi^j(v)$ as the number of type $j$ children of vertex $v$. 

Now consider a colored forest $\f$. Consider the connected components of $\cl^{-1}(j)$ which will be written as $\tFrak^j_0,\tFrak^j_1,\dotsm$. These connected components are trees, and are finite by assumption (C2). Moreover, for each component $\tFrak^j_\ell$ there is some unique vertex $v_\ell^j$ of minimal height, else we would be able to construct a cycle in the forest $\f$. We say that these vertices $v^j_\ell$'s are the $j$-roots of $\f$, and that $v^j_\ell$ is the root of $\tFrak^j_\ell$. We assume that we have indexed these vertices $v^j_\ell$ in a breadth-first manner so that if $\ell<m$ then $v_\ell^j \prec v_m^j$. This can be done since the forest $\f$ is locally finite, and so there are only a finite number of vertices in $\f$ of height at most $h$ for any $h\ge 0$.

Lastly, since $\tFrak^j_\ell$ are finite we can describe a depth-first ordering of the tree in the obvious way. For $\#\tFrak_\ell^j = n$ the depth-first ordering of $\tFrak_\ell^j$ is $w_0,\dots, w_{n-1}$ where $w_0 = v^j_\ell$ and given $w_1,\dots, w_m$ the vertex $w_{m+1}$ is 
\begin{itemize}
	\item[$\,$] The least (in terms of the order $\prec$) child of $w_m$ if any; else
	\item[$\,$] the least unexamined child of $\pi(w_m)$ if any; else
	\item[] the least unexamined child of $\pi(\pi(w_m))$ if any; else,
	\item[] and so on.
\end{itemize} With this we can describe the $j^\text{th}$ depth-first
ordering of the entire forest $\f$, which orders all type $j$ vertices of $\f$. Given two type $j$ vertices $v,w$ we say $v\prej w$ if
\begin{itemize}
	\item[] $v,w\in \tFrak^j_\ell$ and $v$ appears before $w$ in the depth-first ordering of $\tFrak^j_\ell$; or
	\item[] $v\in \tFrak^j_\ell$ and $w\in\tFrak^j_m$ with $\ell<m$. 
\end{itemize} Since the trees are locally finite and each tree $\tFrak^j_\ell$ has finite cardinality, we can enumerate all type $j$ vertices in $\f$ by $w^j_0,w^j_1,\dotsm$ where $w^j_\ell\prej w^j_m$ iff $\ell<m$. This means that the order of these trees appear is in a breadth-first manner; however, within these trees $\tFrak^j_\ell$, the vertices are labeled in a depth-first manner. See Figure \ref{fig:dforderingall} for a depiction of this ordering.

\begin{figure}
	\centering
	\includegraphics[width=\linewidth]{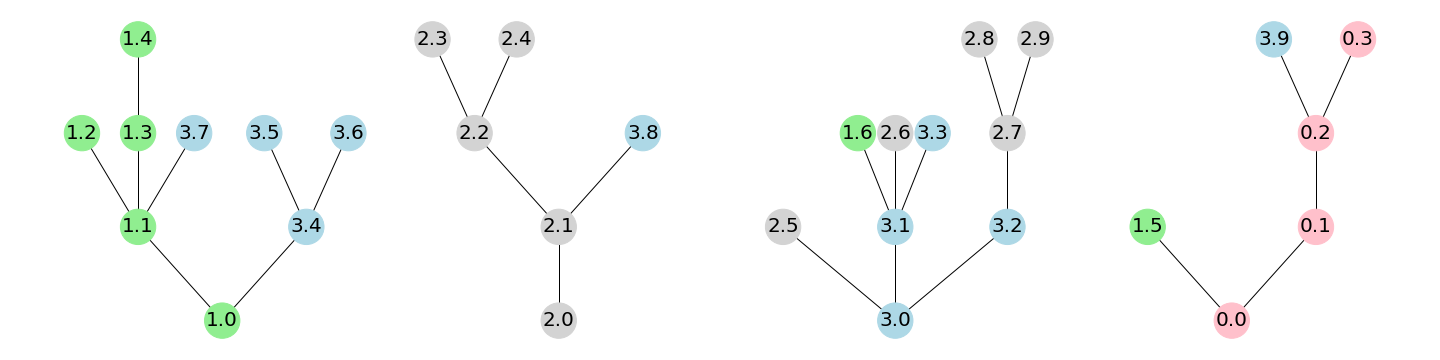}
	\caption{A depiction of the depth-first labeling for each type. Instead of writing $w^j_i$ we write $j.i$. Note that we explore the all the type 1 (in light green) vertices which are in a connected component of $\cl^{-1}(1)$ containing vertex $0$ (now labeled 1.0) before we move to vertex 10 in Figure \ref{fig:bfalllabeling} (labeled 1.5 above).}
	\label{fig:dforderingall}
\end{figure}

\subsection{Processes on Forests}

Given a finite tree $\tFrak$, with root $\rho$ and depth-first ordering $(w_j; 0\le j<\#\tFrak)$, we define the height process of $\tFrak$ as
$$
H_\tFrak(n) = \operatorname{dist}(w_n,\rho).
$$ Clearly, $H_\tFrak$ uniquely characterizes the tree $\tFrak$. 

We also encode the information of the tree $\tFrak$ in another way, called the {\L}ukasiewicz path of $\tFrak$. If let $\chi(w)$ denote the number of children of a vertex $w$. The {\L}ukasiewicz path of a tree $\tFrak$ is denoted $(D_\tFrak(k);k=0,\dots, \#\tFrak)$ and is defined as 
\begin{equation*}
	D_\tFrak(0) = 0,\qquad D_\tFrak(k+1) = D_\tFrak(k)+ \chi(w_k)-1.
\end{equation*} Observe that $\sum_{k=0}^{\#\tFrak-1}\chi(w_k) = \#\tFrak-1$, because both sides count the number of non-root vertices in the tree $\tFrak$. From here, it is not too hard to see that $D_\tFrak(k)>-1$ for all $k=0,\dots, \#\tFrak-1$ and $D_\tFrak(\#\tFrak) = -1$. We recall from \cite{LL_BPLP} without proof that the height process of a tree can be recovered from the {\L}ukasiewicz path by
\begin{equation}\label{eqn:hd}
	H_\tFrak(k) = \#\left\{0\le \ell<k: D_\tFrak(\ell) = \inf_{\ell\le i\le k}D_\tFrak(i)\right\}.
\end{equation}

Now given a colored forest $\f$, we fix a $j\in\{1,2,\dotsm,N\}$. We let $(\tFrak^j_\ell;\ell\ge 1)$ be the connected components of $\cl^{-1}(j)$ constructed in the previous subsection. The roots of the trees $\tFrak^j_\ell$ are denoted by $v^j_\ell$ and the $j^\text{th}$ depth-first ordering of all vertices in $\cl^{-1}(j)$ is given by $w^j_0,w^j_1,\dotsm$. We define the \textit{$j^\text{th}$ height process} of the forest $\f$ as the process $H^j_\f = (H^j_\f(i); \ell\ge 0)$ where 
\begin{equation*}
	H^j_\f(i) = \text{dist}(w_i^j, v^j_\ell),\quad \text{and}\quad \text{ when }w_i^j\in \tFrak^j_\ell.
\end{equation*} We also define the \textit{$j^\text{th}$ {\L}ukasiewicz path} of $\f$ as the process $D^j_\f = (D^j_\f(i); i\ge 0)$ where $$
D^j_\f(0) = 0\qquad D^j_\f(k+1) = D^j_\f(k)+\chi^j(w^j_k)-1,
$$ where, as we recall, $\chi^j(w)$ is the number of type-$j$ children of the vertex $w$. The process $H^j_\f$ is simply the concatenation of all of $H_{\tFrak^j_\ell}$, while the excursions of $D_\f^j$ above its running minimum are the paths $D_{\tFrak^j_\ell}$. Indeed, if $n_p^j = \#\tFrak^j_0+\dotsm+ \#\tFrak_p^j$, then it is easy to see that for $i<\#\tFrak^j_{p+1}$ then
$$
H^j_\f(n_p^j+i) = H_{\tFrak^j_{p+1}}(i),\qquad D_\f^j(n_p^j+i) = D_{\tFrak^j_{p+1}}(i)-(p+1).
$$ From here it follows that \eqref{eqn:hd} remains true with $H_\tFrak$ (resp. $D_\tFrak$) replaced by $H^j_\f$ (resp. $D^j_\f$). We call the vector-valued processes $H_\f = (H^1_\f,\dotsm,H^N_\f)$ the \textit{height process} of $\f$ and $D_\f:=(D^1_\f,\dotsm,D^N_\f)$ the \textit{{\L}ukasiewicz path} of $\f$.

We next define the \textit{height profile} of a colored forest $\f$. This is the $\mathbb{Z}^N$-valued process $Z_\f= (Z^1_\f,\dotsm, Z^N_\f)$ defined by
\begin{equation*}
	Z^j_\f(h) = \{v\in \f: \cl(v) = j, \text{ and }\h(v;\f) = h\}.
\end{equation*} We also wish to keep track of who the parents of these type $j$ vertices are. We keep track of the vertices of type $j$ whose parents are of type $i$ by defining the \textit{$(i\to j)$-height profile} for $i\in\{0,1,\dotsm,N\}$ and $j\in \{1,\dotsm, N\}$ as the process $Z^{i\to j}_\f:=(Z^{i\to j}_\f(h);h\ge0)$ defined by
\begin{equation*}
	Z^{i\to j}_\f(h) = \{v\in \f: \cl(v) = j, \cl(\pi(v)) = i, \h(v;\f) = h\}, 
\end{equation*} with the understanding that if $v$ is a root of $\f$, then $\cl(\pi(v)) = 0$. Thus, we can see for each $j\in\{1,\dotsm, N\}$ and $h\ge 0$
\begin{equation*}
	Z^{j}_\f(h) = \sum_{i=0}^N Z^{i\to j}_\f(h),
\end{equation*} since each type $j$ individual has a parent of some type $i$.

We now define the \textit{cumulative height profiles} as the processes $C^{i\to j}_\f = (C^{i\to j}_\f(h);h\ge 0)$ and $C^j_\f = (C^j_\f(h);h\ge0)$ by
\begin{equation*}
	C^{i\to j}_\f(h) = \sum_{m=0}^h Z^{i\to j}_\f(m),\qquad C^j_\f(h) = \sum_{i=0}^N C^{i\to j}_\f(h) = \sum_{m=0}^h Z^j_\f(m).
\end{equation*}

We also introduce a process which counts the number of type $j$-vertices whose parent is not of type $j$ which have height at most $h$. We denote this by $I^j_\f = (I^j_\f(h);h\ge 0)$. We can see that as \begin{equation}\label{eqn:iDef}
	I^j_\f(h) =\sum_{i\neq j} C^{i\to j}_\f(h).
\end{equation} This process is a convenient way to keep track of all the trees $\tFrak_\ell^j$, which are connected components of the subgraph induced by $\cl^{-1}(j)$ whose root $v^j_\ell$ has height $\h(v^j_\ell)\le h$. Indeed, all of these vertices $v^j_\ell$ are type $j$-vertices whose parent $\pi(v^j_\ell)$ is not of type $j$ and that, moreover, all such vertices correspond to some $v^j_\ell$. Therefore, $\{v^j_\ell; \ell = 0,1,\dotsm, I^j_\f(h)-1\}$ are all the $j$-roots of height at most $h$.

Altering some of the descriptions and definitions in \cite[Pages 111-112]{Duquesne_CRTI} we define the \textit{$j^\text{th}$ left height process} $\cev{H}^j_\f = (\cev{H}^{j}_\f(i);i\ge0)$ by
\begin{equation*}
	\cev{H}^j_\f(i) = \h(w^j_i;\f).
\end{equation*}
We can see that if $w^j_i\in \tFrak^j_\ell$, then the unique path of minimal distance from $w^j_i$ to any root in the forest $\f$ must contain the vertex $v^j_\ell$, and hence we get $$
\cev{H}^j_\f(i) = H^j_\f(i)+\h(v^j_\ell;\f).
$$ 
Since the vertices $v_\ell^j$ for $\ell\in \{0,1,\dotsm, I_\f^{j} (h)-1\}$ enumerate all type-$j$ roots in $\f$ of height at most $h$, we get
\begin{equation*}
	\h(v_\ell^j;\f) = \inf\{h\ge 0: I_\f^j(h)>\ell\}.
\end{equation*}

Next, the $\ell$ for which $w^j_i\in \tFrak^j_\ell$ can be found from the {\L}ukasiewicz path via
\begin{equation*}
	\ell = -\underline{D}^j_\f(i):= - \inf\{D_\f^j(m); m\le i\}.
\end{equation*} Indeed, $\underline{D}_\f^j$ is simply the running minimum of $D_\f^j$ and decrements by 1 each time a tree $\tFrak_\ell^j$ is finished being explored. Hence
\begin{equation}\label{eqn:cevHf}
	\cev{H}^j_\f(i) = H^j_\f(i)+ \inf\{ h\ge 0: I^j_\f(h)> -\underline{D}^j_\f(i)\}.
\end{equation}

The last thing we describe is the breadth-first children functions for the forest $\f$. We first, let $(\bar{w}^i_\ell;\ell\ge 0)$ denote the breadth-first labeling of all vertices of type $i$. That is $\bar{w}_\ell^i$ is a labeling of all type $i$ vertices in $\f$ in a way which preserves the ordering $\prec$: $\bar{w}_\ell^i \prec \bar{w}_m^i$ if and only if $\ell<m$. We define the \textit{breadth-first $i\to j$ children function} by
\begin{equation}\label{eqn:Xijf}
	X^{i,j}_\f(\ell) = \sum_{m=0}^{\ell-1} (\chi^{j}(\bar{w}^i_m) - 1_{[i=j]})\qquad i,j\neq 0
\end{equation}
and 
\begin{equation}\label{eqn:Yjf}
	Y^{j}_\f(\ell) = \sum_{m=0}^{\ell} \chi^{j}(\bar{w}^0_m),\qquad i = 0.
\end{equation} We now observe, from \cite[pg. 1282-1283]{CPU_affine}, that $Z_\f$ is the discrete solution to
\begin{equation}\label{eqn:dtc}
	Z^j_\f(h+1) = Z^j_\f(0)  + \sum_{i=1}^N X^{i,j}_\f\circ C^i_\f(h) + Y^j_\f(h),\qquad C_\f^j(h) = \sum_{\ell=0}^h Z^j_\f(\ell).
\end{equation} Indeed, this can be seen by backwards induction $h$:
\begin{equation*}
	\begin{split}
		Z_\f^j(h+1) &= \bigg(\sum_{i = 1}^N \sum_{\substack{\text{type }i\\\text{ individuals, }v,\\\text{at height }h}} \chi^j(v) \bigg)+ \chi^j(\bar{w}^0_h)\\
		& = Z_\f^j(h) + \bigg(\sum_{i = 1}^N \sum_{\substack{\text{type }i\\\text{ individuals, }v,\\\text{at height }h}} (\chi^j(v) -1_{[i=j]})\bigg)+ \chi^j(\bar{w}^0_h)\\
		&\qquad\qquad\qquad \vdots\\
		&= Z_\f^j(0) + \bigg( \sum_{i=1}^N \sum_{\substack{\text{type }i\\\text{ individuals, }v,\\\text{at height }\le h}} (\chi^j(v) -1_{[i=j]})\bigg) + \sum_{m = 0}^h\chi^j(\bar{w}^0_m)\\
		&= Z_\f^j(0) + \sum_{i=1}^N X_\f^{i,j}\circ C_\f^{i}(h) +Y^j_\f(h),
	\end{split}
\end{equation*} where last equality follows from the observation that there are $C^i_\f(h)$ many type $i$ individuals at height at most $h$, and each of these is labeled by $w_0^i,\dotsm, w_m^i$ where $m = C^i_\f(h)-1$.

\subsection{Randomizing the Model}\label{sec:forestRandom}

We now introduce some randomization into the model described above. We first fix $N(N+1)$ probability measures on $\N_0 = \{0,1,\dotsm\}$, which, a priori, have no assumptions. These measures are labeled $\mu^{i,j}$ and $\nu^j$ for $i,j\in \{1,\dotsm,N\}$. We describe how to construct a colored forest $\f$, by describing the roots and then growing the forest layer-by-layer. We won't focus on the labeling of vertices by elements of $\N$.

We start by fixing a vector $\vec{k} = (1,k_1,\dotsm, k_n)$. This will describe the roots of our forest. For ease of notation, we let $\f_{-1} = \emptyset$. 
\begin{enumerate}
	\item There will be $k_j$ roots of color $j$ at height 0, and 1 root of color 0 at height 0, labeled subject to (O1). Call this $\f_0$
	\item For each $h\ge 0$, and $w\in \f_{h}\setminus \f_{h-1}$ of color $i$ generate, for each $j\in [N]$, independent random variables $\chi^{j}(w)$ with distribution $\mu^{i,j}$ (when $i\neq 0$) or $\nu^j$ (when $i=0$). If $i= 0$ then generate 1 vertex of type 0 as well. 
	
	\item For all $j\in [N]$ and $w\in \f_{h}\setminus \f_{h-1}$. At height $h+1$ add $\chi^j(w)$ children of type $j$ and parent $w$, labeled subject to (C3). Call the resulting forest $\f_{h+1}$.
	\item Continue this process ad infinitum. 
\end{enumerate}
The resulting random colored forest will be defined as $\f= \cup_{h\ge 0} \f_h$. 

We call the resulting forest a multitype Galton-Watson immigration forest with offspring distributions $\boldsymbol{\mu} = (\mu^{i,j}; i,j\in [N])$ and immigration $\boldsymbol{\nu} = (\nu^j; j\in [N])$ started from $\vec{k}$ individuals, which is abbreviated $\GWI_{\vec{k}}(\boldsymbol{\mu},\boldsymbol{\nu})$. For more information on multitype Galton Watson processes see \cite{W_MCBI, DFS_Affine, CPU_affine} and references therein. We can easily see that the height profile $Z_\f$ of $\f$ is a multitype Galton-Watson process with immigration. Indeed, if $(\xi^{i,j,h}_\ell; h \ge 0, \ell \ge 1)$ are i.i.d. with common distribution $\mu^{i,j}$ and $(\eta^{j}_h; h\ge 0)$ are i.i.d. with common distribution $\nu$, then conditionally on $Z_\f(h) = (z_1,\dotsm,z_N)$ we have
$$
Z^j_\f(h+1) =\eta^j_h+ \sum_{i=1}^N \sum_{\ell=1}^{z_i} \xi^{i,j,h}_\ell.
$$

We make the following crucial observation, for any $\GWI_{\vec{k}}(\boldsymbol{\mu},\boldsymbol{\nu})$ forest then
\begin{equation*}%\label{eqn:dx}
	\left(D_\f^j(m); m\ge 0 \right) \overset{d}{=} \left(X^{j,j}_\f(m);m\ge 0  \right).
\end{equation*}
Indeed, this follows from the observation that both $(\chi^j(\bar{w}^j_m);m\ge 0)$ and $(\chi^j(w^j_m);m\ge0)$ are both sequences of independent random variables with common distribution $\mu^{j,j}$.

As can be seen from equations \eqref{eqn:Xijf} and \eqref{eqn:Yjf}, the processes $X_\f^{i,j}$ and $Y^j_\f$ are random walks the the solutions $Z_\f^j$ solves (the now stochastic) equation \eqref{eqn:dtc}.

\section{Overview of Branching Processes and Height Processes}\label{sec:bp}

\subsection{Continuous State Branching Processes}
Continuous state branching processes with or without immigration are an object of much study. Continuous state branching (CB for short) processes are Feller processes on $[0,\infty]$ with cemetery states of $0$ and $\infty$. An immigration component can be added to obtain a continuous state branching process with immigration (CBI for short) which is a Feller process on $[0,\infty]$ with only $\infty$ as the only absorbing state, excluding the situation where the immigration rate is 0. Kawazu and Watanabe in \cite{KW_BPI} show that $\CBI$ processes are uniquely determined by their Laplace transforms, i.e. if $\Z$ is a $\CBI$ process then there exists functions $\psi$ and $\phi$, for all $\lambda>0$
\begin{equation*}
	-\log \E_x \left[\exp\{-\lambda \Z_t\} \right] = xu(t,\lambda)+ \int_0^t \phi(u(s,\lambda))\,ds
\end{equation*} where
$u$ is the unique solution to the integral equation
\begin{equation*}
	u(t,\lambda)+ \int_0^t \psi(u(s,\lambda))\,ds = \lambda.
\end{equation*} The function $\psi$ is called the \textit{branching mechanism} and the function $\phi$ is called the \textit{immigration rate}. We say that $\Z$ is a $\CBI(\psi,\phi)$ process for short, and if we wish to specify the starting position we will write $\CBI_x(\psi,\phi)$. The functions $\psi$ and $\phi$ must satisfy \cite{KW_BPI,S_LT}
\begin{equation}\label{eqn:pp1}
	\begin{array}{l}
		\displaystyle\psi(\lambda) = -\kappa +\alpha \lambda + \beta \lambda^2 + \int_{(0,\infty)} \left(e^{-\lambda r} - 1 + \lambda r 1_{[r<1]} \right)\,\pi(dr)\\
		\displaystyle\phi(\lambda) = \kappa' + \alpha' \lambda - \int_{(0,\infty)} \left(e^{-\lambda r}-1 \right)\, \bar{\pi}(dr)
	\end{array},
\end{equation}
where $\kappa, \beta, \kappa',\alpha' \ge 0$, $\alpha \in \R$, $\pi$ and $\bar{\pi}$ are Radon measures on $(0,\infty)$ with $\displaystyle\int_{(0,\infty)} (1\wedge r^2)\,\pi(dr)<\infty$ and $\displaystyle\int_{(0,\infty)} (1\wedge r)\,\bar{\pi}(dr)<\infty$. 

The description in \eqref{eqn:pp1} gives a bijective relationship between CBI processes and certain pairs of L\'{e}vy processes. That is, there exists a spectrally positive (i.e. no negative jumps) L\'evy process $\X$ and a subordinator $\Y$ such that for each $\lambda>0$
\begin{equation*}%\label{eqn:levy}
	\E\left[\exp\left\{-\lambda \X_t \right\}\right] = \exp\left\{t\psi(\lambda) \right\},\qquad
	\E\left[\exp\left\{-\lambda \Y_t \right\}\right] = \exp\left\{-t\phi(\lambda) \right\}.
\end{equation*}
For more information of L\'evy processes, see, for example, Bertoin's monograph \cite{Bertoin_LP}. 

There does exist a path-wise relationship between $\CBI$ processes and these L\'evy processes $\X$ and $\Y$, which is due to Caballero, P\'erez Garmendia and Uribe Bravo in \cite{CGU_CSBPI}. Their work generalized the result accredited to Lamperti \cite{Lamperti_CSBP1}, but first proved by Silverstein in \cite{S_LT}. Given independent $\X$ and $\Y$, there exists a unique c\`adl\`ag solution to the integral equation
\begin{equation*}
	\Z_t = x+\X_{\int_0^t \Z_s\,ds}+\Y_t,
\end{equation*} and, moreover, $\Z$ is a $\CBI_x(\psi,\phi)$ processes. When $\Y$ is zero the process $\X$ is stopped upon hitting level $-x$ and, in this case, the map is invertible.

\subsection{Multi-type Branching Processes}\label{sec:mbp}

There are numerous generalizations of $\CBI$ processes, including allowing immigration from outside sources. We call these new processes multitype continuous state branching processes (resp. with immigration), written as $\MCB$ (resp. $\MCBI$). These were described in a two-dimensional system in \cite{W_MCBI}. A more general picture of multitype branching processes was described in \cite{DFS_Affine} as a particular example of so-called {affine processes}.

An $\R_+^N$-valued Markov process $\Z$ is a multitype continuous state branching process with immigration if, for each $\lambda \in \R_+^N$ and $x\in \R_+^N$ 
\begin{equation*}%\label{eqn:mcbiLT1}
	\E_x\left[\exp\left\{ -\langle\lambda,\Z_t\rangle\right\} \right] = \exp\left\{-\langle x,u(t,\lambda)\rangle -\int_0^t \tilde{\phi}(u(s,\lambda))\,ds \right\},
\end{equation*} for a particular function $\phi$ and
where $u = (u_1,\dotsm,u_N)$ is a solution to
\begin{equation*}%\label{eqn:riccati}
	u_i(t,\lambda)+ \int_0^t \tilde{\psi}_i(u(s,\lambda))\,ds = \lambda_i
\end{equation*} for a collection of functions $\tilde{\psi}_i$.

As can be observed from Theorem 2.7 in \cite{DFS_Affine}, there is a bijection between MCBI processes $\Z$ and a collection of $(N+1)$ $\R^N$-valued L\'evy processes $\X^i$, $\Y$ for $i\in [N]$. That is, the functions $\tilde{\psi_i}$ and $\tilde{\phi}$ are related to $\X^i$ and $\Y$ by 
\begin{equation*}%\label{eqn:levyLT}
	\E\left[\exp\left\{-\langle\lambda, \X^i_1\rangle \right\} \right] = e^{\tilde{\psi}_i(\lambda)}\qquad \E\left[\exp\left\{-\langle\lambda, \Y_1\rangle \right\} \right] = e^{-\tilde{\phi}(\lambda)}.
\end{equation*}

However, the authors of \cite{DFS_Affine} do not give a path-wise relationship. There is a path-wise representation, due to Caballero, P\'erez Garmendia and Uribe Bravo in \cite{CPU_affine}, which extended the result in a Ph.D. thesis of Gabrielli \cite{Gab_Affine} which had some additional technical assumptions. In the former work, the authors show that given $(N+1)N$ not necessarily independent L\'evy processes $\X^{i,j}$ and $\Y^j$ such that $\X^{j,j}$ is spectrally positive and the rest of subordinators, then there exists a solution $\Z$ to the following initial value problem
\begin{equation}\label{eqn:lampertiGen}
	\Z^j_t = x_j + \sum_{i=1}^N \X^{i,j}({\C^i_t}) + \Y_t^j,\qquad \C^j_t = \int_0^t \Z^j_s\,ds,
\end{equation} where $\X^{i} = (\X^{i,1},\X^{i,2},\dotsm, \X^{i,N})$ and $\Y = (\Y^1,\Y^2,\dotsm,\Y^N)$ are $\R^N$-valued L\'{e}vy processes. As the notation may suggest, the process $\Z_t = (\Z^1_t,\dotsm, \Z^N_t)$ is a multitype continuous state branching process associated with the $(N+1)$ L\'evy processes. The work of \cite{CPU_affine} shows that for any MCBI process, we can find a decomposition of the form \eqref{eqn:lampertiGen}. Observe that \eqref{eqn:lampertiGen} the continuous time analog of equation \eqref{eqn:dtc}.

\subsection{Height Processes}\label{sec:hp}

The continuous time height process $\HH = (\HH_t;t\ge 0)$ is the continuous time analog of equation \eqref{eqn:hd} where the random walk $D_\f$ is replaced with a spectrally positive L\'evy process $\X = (\X_t;t\ge 0)$ with Laplace exponent $(-\psi)$. Such a height process with nice properties exists under restrictions on what type of L\'{e}vy processes we consider.

To get these nice properties, we will assume that $\psi$ satisfies equation \eqref{eqn:psijs}, where $\psi_j$ is replaced with $\psi$ (along with corresponding replacements for $\alpha,\beta,\pi$).
The integral condition guarantees the almost sure extinction of a $\CBI(\psi,0)$ process (see \cite{Grey_explosion}) and implies either $\beta>0$ or $\int_0^1 r\,\pi(dr) = \infty$ and, hence, that $\X$ has paths of infinite variation almost surely. 

The analog of \eqref{eqn:hd} is the $\psi$-height process $\HH$, which is associated with $X$, is defined to give a meaningful measure to the set in \eqref{eqn:ltForH}. See Section 1.2 of \cite{DL_Levy} for a precise construction. We do state, however, that when $\beta>0$ the height process is
\begin{equation*}
	\HH_t = \frac{1}{\beta}\text{Leb}\left\{ \inf_{s\le r\le t}\X_r; s\le t \right\}.
\end{equation*} Moreover, under the conditions on $\psi$ in \eqref{eqn:psijs}, the height process $\HH$ is continuous, see \cite[Theorem 1.4.3]{DL_Levy}.

The Ray-Knight theorem \cite[Theorem 1.4.1]{DL_Levy} for the process $\HH$ is now briefly recalled. The local time $\LL = (\LL^a_t;a,t\ge 0)$ of $\HH$, is defined by the approximation formula (\cite[Proposition 1.3.3]{DL_Levy})
\begin{equation*}
	\lim_{\eps\downto 0} \E\left[\sup_{s\in[0,t]} \left|\frac{1}{\eps}\int_0^s 1_{[a<\HH_s\le a+\eps]} - \LL^a_t\right| \right].
\end{equation*} For $\tau_r = \inf\{t> 0: \X_s > -r\}$, the Ray-Knight theorem states that $\Z = (\Z_a := \LL^a_{\tau_r};a\ge 0)$ is a $\CBI(\psi,0)$ process started from the value $r$. The work was extended by Duquesne \cite{Duquesne_CRTI} and Lambert \cite{Lambert_CBI} to include some immigration mechanism.

\section{The Weak Solution}\label{sec:ws}

In this section we describe how we construct solutions to the stochastic equation in \eqref{eqn:HSDE}, under the assumptions of Theorem \ref{thm:existH}. We will prove the theorem, by appealing to several lemmas which follow from results in the existing literature.

\subsection{Preliminary Lemmas}\label{sec:pl}

We now describe the assumptions we will make on the $\GWI$ forests we analyze. We suppose that $(\bmu_p;p\ge 1)$ and $(\bnu_p;p\ge 1)$ are a sequence of probability measures and $(\gamma_p;p\ge 1)$ are any increasing sequence of non-negative integers. For each $i,j\in [N]$ and $p\ge 1$ we let $(\xi^{i,j,p}_\ell;\ell\ge 0)$ (resp. $(\eta^{j,p}_\ell;\ell\ge 0)$) denote i.i.d. sequences with common distribution $\mu^{i,j}_p$ (resp. $\nu^j_p$). We let $g^{j,p}$ denote the generating function of $\mu^{j,j}_p$ and iteratively define $g_n^{j,p} = g^{j,p}_{n-1}\circ g^{j,p}$ where $g_0^{j,p} = id$. For an real number $x$ let $[x]$ denote the greatest integer smaller than $x$, i.e. the integer part of $x$. We make the following assumptions on these measures:
\begin{itemize}
	\item[(\textbf{A1})]\label{ass:1} Jointly for all $i,j\in [N]$ we have the following convergence in $\D(\R_+,\R)$ 
	\begin{align*}
		\displaystyle\left(\frac{1}{p} \sum_{\ell=0}^{[p\gamma_pt]-1} (\xi^{i,j,p}_\ell - 1_{[i=j]});t\ge0\right) &\Longrightarrow \left(\X^{i,j}_t;t\ge 0\right)\\
		\left(\frac{1}{p} \sum_{\ell=0}^{[\gamma_p t]-1} \eta^{j,p}_\ell;t\ge 0 \right) &\Longrightarrow \left( \Y^{j}_t;t\ge 0\right)
	\end{align*}
	\item[(\textbf{A2})] The processes $\X^{i,j}_t$ and $\Y^j_t$ above are independent L\'evy processes where $\Y^j_t = \delta_j t$ and for $i\neq j$, $\X^{i,j}_t = \alpha_{i,j} t$ for $\delta_j>0$ and $\alpha_{i,j}\ge 0$. The Laplace exponents $\psi_j$ of $\X^{j,j}$ satisfy \eqref{eqn:psijs}.
	\item[(\textbf{A3})] The generating functions satisfy $$\liminf_{p\to \infty}g_{[\delta \gamma_p]}^{j,p}(0)>0\qquad\forall \delta>0.$$
\end{itemize}
If each assumption above is satisfied, we say that $\bmu_p$ and $\bnu_p$ satisfy assumption (\textbf{A}).

We remark that if we are given Laplace exponents $(\psi_j;j\in[N])$ which are admissible then we can construct a sequence of probability measures $\bmu_p$ and $\bnu_p$ which satisfy assumption (\textbf{A}). Indeed, the the existence of measure $(\mu^{j,j}_p;p\ge 1)$ is essentially the content of the the functions $(\psi_j;j\in[N])$ being admissible. To get the remaining measures we can take, for example, $(\xi^{i,j,p}_\ell;\ell\ge 0)$ for $i\neq j$ and $(\eta^{j,p}_\ell)$ to be independent Poisson random variables with appropriate means. 

Throughout the sequel we will write the subscript $p$ for processes on forests as opposed to $\f_p$. For example, we will write $Z^j_p(h)$ as opposed to $Z^j_{\f_p}(h)$. With this new notation, we can state the following lemma. A rigorous proof is omitted since it follows from equation \eqref{eqn:dtc} along with Theorem 2 and Theorem 3 in \cite{CPU_affine}.

\begin{lem}\label{lem:discZ}
	Suppose $\boldsymbol{\mu}_p$ and $\boldsymbol{\nu}_p$ satisfy assumption (\textbf{A}). Suppose that $\vec{k}_p = (1,k_{1,p},\dots,k_{N,p})$ where $k_{j,p}/p\to x_j\ge 0$ as $p\to\infty$. Let $(\f_p;p\ge 1)$ denote a sequence of $\GWI_{\vec{k}_p}(\bmu_p,\bnu_p)$ forests. Then, in the Skorokhod $J_1$ topology on $\D(\R_+,\R^N)$ the following convergence holds 
	\begin{equation}\label{eqn:zconv}
		\left(\frac{1}{p}Z_p([\gamma_pv]); v\ge 0 \right) \Longrightarrow \left(\mathbf{Z}_v; v\ge 0 \right),
	\end{equation} where $\mathbf{Z} = (\mathbf{Z}^1,\dotsm, \mathbf{Z}^N)$ is the unique solution to 
	\begin{equation}\label{eqn:Ztc}
		\Z^j_v = x_j + \sum_{i=1}^N \X^{i,j}_{\C^i_v} + \Y^j_v ,\qquad \C^j_v = \int_0^v \Z^j_s\,ds.
	\end{equation} Moreover, the convergence is joint with the convergence in assumption (\textbf{A1})
\end{lem}

For a function $f:\R_+\to\R$ we write $\underline{f}(t) = \inf_{s\le t} f(s)$. We now observe that the height process for the forest, along with the {\L}ukasiewicz path converge jointly as well. We state this as the following lemma, which is a simple application of Corollary 2.5.1 in \cite{DL_Levy} and using the fact \cite[Equation (1.7)]{DL_Levy} holds in our situation by (\textbf{A3}). 
\begin{lem}\label{lem:dhd}
	Suppose the conditions in Lemma \ref{lem:discZ} on $\boldsymbol{\mu}_p$, $\boldsymbol{\nu}_p$ and $\vec{k}_p$ hold. Then, in the $\D(\R_+,\R^3)$ the following joint convergence holds 
	\begin{equation*}%\label{eqn:jointDHD}
		\left(\left(\frac{1}{p} D_p^j([p\gamma_pt]), \frac{1}{\gamma_p}H^j_p([p\gamma_pt]), \frac{1}{p}\underline{D}_{p}^j([p\gamma_pt])\right);t\ge 0 \right)\Longrightarrow \left(\left(\tilde{\X}^{j,j}_t ,\HH^j_t, \inf_{s\le t} \tilde{\X}^{j,j}_s \right);t\ge 0\right)
	\end{equation*} where $\tilde{\X}^{j,j} \overset{d}{=} \X^{j,j}$ and $\HH^j_t$ is a $\psi_j$-height process constructed from $\tilde{\X}^{j,j}$.
\end{lem}

\begin{remark}
	We note that by \cite[Lemma 1.3.2]{DL_Levy} that $-\inf_{s\le t} \tilde{\X}^{j,j}_s = \ell_t^j$ where $\ell^j_t$ is the local time at level 0 and time $t$ of the process $\HH^j$ defined as the $L^1$-limit in \eqref{eqn:ltLim}.
\end{remark}

\subsection{Convergence of the left-height process}\label{subsec:cevH}

We start with a useful convergence lemma, along with an observation on the path-wise behavior of the limiting process.
\begin{lem} \label{lem:ulim} Suppose $\boldsymbol{\mu}_p, \boldsymbol{\nu}_p$ satisfy assumption (\textbf{A}) and the $\vec{k}_{p} = (1,k_{1,p},\dotsm,k_{N,p})$ satisfies $k_{j,p}/p \to x_j\ge 0$ as $p\to\infty$.  Then the following convergence holds, jointly with the convergences in equations \eqref{eqn:zconv} and assumption (\textbf{A1}),
	\begin{equation}\label{eqn:ulim}
		\left(\frac{1}{p} I_p^j([\gamma_pv]);v \ge 0 \right) \Longrightarrow \left(\U_v^j ; v\ge0  \right),
	\end{equation} where 
	$$\U^j_v = x_j+ \sum_{i\neq j} \alpha_{i,j} \C_v^i + \delta_jv,\qquad \C^j_v = \int_0^v \Z^j_s\,ds.	
	$$
	Moreover, the process $\U^j$ is strictly increasing and $\U^j_t\to \infty$ as $t\to\infty$.
\end{lem}
\begin{proof}
	We begin by noting that 
	\begin{equation}\label{eqn:discCtc}
		C^{i\to j}_p(h) = \left\{ \begin{array}{ll}
			X^{i,j}_p\left(C^i_p(h-1)\right) &: i\neq 0, j\\
			k_{j,p}+Y^{j}_p(h) &: i = 0
		\end{array}\right..
	\end{equation} Indeed, all type $i$ vertices of height at most $h-1$ are enumerated by $w^i_\ell$ for $\ell = 0,1,\dots, C^i_p(h-1)-1$. We get the above discrete time change observation by noting that every vertex of type $j$ and height at most $h$ as a parent of height at most $h-1$. This is essentially the same backwards induction argument used in the justification of \eqref{eqn:dtc}.
	
	We now observe the following
	\begin{align*}
		\frac{1}{p\gamma_p} C^j_p([\gamma_pv]) = \frac{1}{p\gamma_p} \sum_{h=0}^{[\gamma_pv]}Z^{j}_p(h) &= \int_0^{[\gamma_pv]+1} \frac{1}{p\gamma_p}Z_p^j([s])\,ds\\
		&=  \int_0^{([\gamma_p v]+1)/\gamma_p} \frac{1}{p} Z_p^j([\gamma_ps])\,ds.
	\end{align*} Hence, the we get
	$\displaystyle \left(\frac{1}{p\gamma_p} C^j_p([\gamma_pv]);v\ge 0\right) \Longrightarrow (\C^j_v;v\ge0)$ in $\D(\R_+,\R)$, where $\C^j_v$ is as in \eqref{eqn:Ztc} by the continuous mapping theorem. Moreover, this convergence can easily be seen to hold jointly with the convergences of \eqref{eqn:zconv} and (\textbf{A1}).
	
	The desired convergence in \eqref{eqn:ulim} thus follows from equation \eqref{eqn:discCtc} and standard time-change results. Indeed, the scaling limit of the processes $X_p^{i,j}$ are linear functions and consequently are almost surely continuous. By \cite[Lemma pg. 151]{B_CPM} and the representation in \eqref{eqn:iDef}, the following convergence holds
	\begin{equation*}
		\begin{split}
			\left(\frac{1}{p} I_p^j\left([\gamma_p v]\right);v\ge 0\right) &= \left(\frac{1}{p}\left(k_{j,p}+Y^j_p([\gamma_p v]) +  \sum_{i\neq j} X^{i,j}_p\left(  C_p^i([\gamma_p v]) \right) \right);v\ge 0 \right)\\
			&\Longrightarrow \left(x_j + \delta_j v+  \sum_{i\neq j} \alpha_{i,j} \C^i_v\right).
		\end{split}
	\end{equation*}
	
	The statement about $\U^j$ being strictly increasing follows from the observations that $\U^j$ has differentiable paths almost surely (the processes $\C^i$ are differentiable) and its derivative
	\begin{equation*}
		\frac{d}{dv} \U^j_v = \delta_j+ \sum_{i\neq j} \alpha_{i,j} \Z^i_v>0, \qquad \text{Lebesgue a.e. }v.
	\end{equation*} This follows from
	 $\mathbf{Z}^i\ge 0$ a.s. since $\Z = (\Z^1,\dotsm , \Z^N)$ is a Feller process on $[0,\infty)^N$ and $\delta_j>0$.
\end{proof}

We have now gathered all of the pieces necessary for proving the existence of a solution to equation \eqref{eqn:HSDE}. Instead of stating the proof of Theorem \ref{thm:existH}, we prove the theorem below, which is easily seen to imply both Theorem \ref{thm:existH} and Corollary \ref{cor:bm}.
\begin{thm}\label{thm:cevHconv}
	Suppose $\boldsymbol{\mu}_p, \boldsymbol{\nu}_p$ satisfy assumption (\textbf{A}) and $\vec{k}_p$ satisfy $k_{j,p}/p\to x_j$ as $p\to\infty$.
	Then
	\begin{equation*}
		\left(\left(\frac{1}{\gamma_p} \cev{H}^j_p([p\gamma_pt]);j\in [N], t\ge 0  \right);\left(\frac{1}{p}Z_p^j([\gamma_pv]);j\in [N], v \ge 0  \right) \right)\\
	\end{equation*} along a subsequence converges weakly in the Skorohod space to 
	\begin{equation*}
		\left(\left( \cev{\HH}_t^j; j\in [N],t\ge 0\right) ; \left(\LL_\infty^v(\cev{\HH}^j);j\in [N], v\ge0 \right)\right).
	\end{equation*}
where
	$$
	\cev{\HH}^j_t = \HH^j_t + \inf\left\{x>0 : \U^j_x > \ell^j_{t} \right\}.
	$$
\end{thm}

\begin{remark}
	We observe that Theorem \ref{thm:existH}, and hence Corollary \ref{cor:bm} follows from the above theorem. Indeed, Theorem \ref{thm:existH} follows from the observation that $U_0^j = x_j$ and $\frac{d}{dx} \U_x^j = \delta_j + \sum_{i\neq j} \lambda_{i,j} \LL_\infty^x(\HH^j)$. Hence by \cite[Proposition 0.4.6]{RY} the process
	$$
	\J^j_t := \inf\left\{x>0 : \U^j_x> \ell_{t}^j \right\},
	$$ must satisfy
	$$
	\J^j_t = \int_{\tau_{x_j}^j}^{t\vee\tau_{x_j}^j} \frac{1}{\delta_j + \sum_{i\neq j} \alpha_{i,j} \LL_\infty^{\J^j_s}(\HH^i)} \,d\ell_s^j,\qquad \tau_{x}^j = \inf\{t: \ell^j_t > x\}.
	$$
	
\end{remark}

\begin{proof}[Proof of Theorem \ref{thm:cevHconv}]
	The proof follows quite easily from the various lemmas we have proved and discrete processes which we have defined. We observe from \eqref{eqn:cevHf} the following holds
	\begin{align*}
		\frac{1}{\gamma_p} \cev{H}^j_p([p\gamma_pt]) &= \frac{1}{\gamma_p}H^j_p([p\gamma_pt]) + \frac{1}{\gamma_p}\inf\left\{ h : I_p^j(h)>  -\underline{D}_p^j([p\gamma_pt])\right\}\\
		&= \frac{1}{\gamma_p} H_p^j ([p\gamma_p t]) + \inf\left\{\frac{h}{\gamma_p} : \frac{1}{p} I_p^j(h) > \frac{-1}{p}\underline{D}_p^j([p\gamma_pt])\right\}
		\\
		&= \frac{1}{\gamma_p}H^j_p([p\gamma_pt]) + \inf\left\{ x : \frac{1}{p}I_p^j([\gamma_px])> -\frac{1}{p} \underline{D}_p^j([p^2t])\right\}.
	\end{align*} 
	We now recall a result of Whitt \cite{Whitt_WC}. It states that on the space $\D^\upto(\R_+,\R_+):= \{g\in \D(\R_+,\R_+): \sup g = \infty\}$, the first passage time is continuous at each strictly increasing function. In terms of functions, the function
	$$
	F:\D^{\upto}(\R_+,\R_+)\to \D(\R_+,\R_+) \qquad\text{ by}\qquad F(f)(t) = \inf\{s>0 : f(s)>t\}
	$$ is continuous at each $f$ which is strictly increasing and diverge to infinity. 
	
	Hence, by Lemma \ref{lem:ulim} we know that $\U^j_x$ is almost surely continuous, diverging towards infinity and is strictly increasing and hence we can conclude in $\D(\R_+,\R_+)$ that
	\begin{equation*}
		\left(\inf\left\{x: \frac{1}{p}I^{j}_p([\gamma_px])>v \right\}; v \ge 0 \right)\Longrightarrow \left(\inf\left\{x: \U^j_x>v \right\}; v \ge 0 \right). 
	\end{equation*} Since $-\underline{D}^j_p$ is non-decreasing, Lemma \ref{lem:dhd} and \cite[Lemma pg. 151]{B_CPM}
	\begin{equation*}
		\left(\inf\left\{ x : \frac{1}{p}I_p^j([\gamma_px])> -\frac{1}{p} \underline{D}_p^j([p^2t])\right\}; t \ge 0 \right) \Longrightarrow \inf\left\{x>0 : \U^j_x> \ell_{t}^j \right\}.
	\end{equation*} By a tightness argument, we can take this convergence to be joint with the convergence of $\frac{1}{\gamma_p}H^j_p([p\gamma_pt])$ towards $\HH^j_t$ at least along a subsequence. Hence we prove the convergence of the left-height process. 
	
	The rescaled process $\frac{1}{p}Z^j_p([\gamma_pv])$ converges by Lemma \ref{lem:discZ}. The joint convergence follows from similar tightness arguments as the proof of Corollary 2.5.1 in \cite{DL_Levy}. See also Theorem 1.5 in \cite{Duquesne_CRTI} for a similar result which relies on the proof of Corollary 2.5.1. This completes the proof of the result.
\end{proof}

\subsection{Consequences of Theorem \ref{thm:cevHconv}}\label{sec:con}

In this subsection we discuss the consequences of Theorem \ref{thm:cevHconv}. We observe that by Lemma \ref{lem:discZ}, Theorem \ref{thm:cevHconv} above and Theorem 1 in \cite{CPU_affine} the local time of the processes $\cev{\HH}^j$ is a multitype continuous state branching process. We state this as the following proposition:
\begin{prop}\label{prop:mcbi}
	The processes $\Z_v^j = \LL_\infty^v(\cev{\HH}^j)$, $j\in[N]$ form a multitype continuous state branching process with immigration determined by equation \eqref{eqn:Ztc}.
\end{prop}

Stochastic equations for $\Z$ are possible thanks to the results of \cite{BLP_SDE}. We now use that work and focus on the Brownian situation of Corollaries \ref{cor:bm} and \ref{cor:zsde}. In particular we assume that $\X^{j,j}$ have Laplace exponents of the form
\begin{equation*}
	\psi_j(\lambda) = -\alpha_{j,j}\lambda + \beta_j \lambda^2
\end{equation*} where $\beta_j>0$ and $\alpha_{j,j}\le 0$. In turn, the process $\X^{j,j}_t = \sqrt{2\beta_j} B^j_t + \alpha_{j,j}t$ is a (constant multiple of a) Brownian motion with negative drift. The computations in Section 5 of \cite{BLP_SDE} imply that
\begin{equation*}
	\Z_v^j = x_j + \int_0^v \left(\delta_j + \sum_{i=1}^N \alpha_{i,j} \Z^i_s\right)\,ds + \int_0^v\sqrt{2\beta_j \Z_s^j }\,dW^j_s, 
\end{equation*} for an $\R^N$-valued Brownian motion $(W^1,\dotsm, W^N)$. Alternatively, this can be seen more directly from equations \eqref{eqn:Ztc} by writing $X^{i,j}_{\C^i_v}$ for $i\neq j$ as an integral and applying the Dambis-Dubins-Schwarz theorem \cite[Theorem V.1.6]{RY} to the the term $\X^{j,j}\circ \C^j(v)$. Indeed, 
\begin{align*}
	\X^{j,j}\circ \C^j(v) &= \sqrt{2 \beta_j} B^j\circ \C^j(v) + \alpha_{j,j} \int_0^v \Z^j_s\,ds\\
	&= \int_0^v \sqrt{2 \beta_j \Z^j_s}\,dW_s^j + \int_0^v \alpha_{j,j}\Z^j_s\,ds.
\end{align*} From here Corollary \ref{cor:zsde} follows. 

\section*{Acknowledgements}

The author was supported by NSF grant DMS-1444084.

% The authors would like to thank ...
%
% The first author was supported by ...
%
% The second author was supported in part by ...

%%%%%%%%%%%%%%%%%%%%%%%%%%%%%%%%%%%%%%%%%%%%%%
%% Supplementary Material, if any, should   %%
%% be provided in {supplement} environment  %%
%% with title and short description.        %%
%%%%%%%%%%%%%%%%%%%%%%%%%%%%%%%%%%%%%%%%%%%%%%
%\begin{supplement}
%\stitle{???}
%\sdescription{???.}
%\end{supplement}

%%%%%%%%%%%%%%%%%%%%%%%%%%%%%%%%%%%%%%%%%%%%%%%%%%%%%%%%%%%%%
%%                  The Bibliography                       %%
%%                                                         %%
%%  imsart-number.bst  will be used to                     %%
%%  create a .BBL file for submission.                     %%
%%                                                         %%
%%  Note that the displayed Bibliography will not          %%
%%  necessarily be rendered by Latex exactly as specified  %%
%%  in the online Instructions for Authors.                %%
%%                                                         %%
%%  MR numbers will be added by VTeX.                      %%
%%                                                         %%
%%  Use \cite{...} to cite references in text.             %%
%%                                                         %%
%%%%%%%%%%%%%%%%%%%%%%%%%%%%%%%%%%%%%%%%%%%%%%%%%%%%%%%%%%%%%

%% if your bibliography is in bibtex format, uncomment commands:
\bibliographystyle{imsart-number} % Style BST file
\bibliography{references}       % Bibliography file (usually '*.bib')

%% or include bibliography directly:
% \begin{thebibliography}{}
% \bibitem{b1}
% \end{thebibliography}

\end{document}